\documentclass[11pt]{article}

\usepackage{amsfonts}
\usepackage{amssymb,amsmath,amsthm}
\usepackage{latexsym}
\usepackage{fullpage}
\usepackage{hyperref}

\newtheorem{theorem}{Theorem}[section]

\newtheorem{prop}[theorem]{Proposition}

\newtheorem{lemma}[theorem]{Lemma}
\newtheorem{cor}[theorem]{Corollary}
\newtheorem{defn}[theorem]{Definition}

\newtheorem{claim}[theorem]{Claim}

\theoremstyle{definition}

\newcounter{tenumerate}

\def\P{\mathbb{P}}

\newcommand{\one}{\1}
\newcommand{\deq}{\stackrel{\scriptscriptstyle\triangle}{=}}

\renewcommand{\epsilon}{\varepsilon}

\newcommand{\1}{\mathbf{1}}

\DeclareMathOperator{\var}{Var}

\newcommand{\R}{{\mathbb R}}
\newcommand{\N}{{\mathbb N}}

\newcommand{\E}{{\mathbb E}}
\newcommand{\remove}[1]{}

\renewcommand{\leq}{\leqslant}
\renewcommand{\geq}{\geqslant}

\def\XXint#1#2#3{{\setbox0=\hbox{$#1{#2#3}{\int}$}
\vcenter{\hbox{$#2#3$}}\kern-.5\wd0}}

\begin{document}

\title{{\bf A sharp estimate for cover times on binary trees}}

\author{Jian Ding\thanks{Partially supported by Microsoft Research. Much of the work was carried out during a visit of J.D. to Weizmann institute.} \\
U. C. Berkeley \and Ofer Zeitouni\thanks{Partially supported by NSF
grant DMS-0804133 and a grant from the Israel Science Foundation.}
\\ University of Minnesota\\ \& Weizmann institute}
\date{March 27, 2011}

\maketitle

\begin{abstract}
    We compute the second order correction for the cover time
    of the binary tree of depth $n$ by (continuous-time) random walk,
    and show that with  probability approaching $1$ as $n$ increases,
    $\sqrt{\tau_{\mathrm{cov}}}=\sqrt{|E|}[
\sqrt{2\log 2}\cdot n - {\log n}/{\sqrt{2\log 2}} + O((\log\log
n)^8]$, thus showing that the second order correction differs from
the corresponding one for the maximum of the Gaussian free field on
the tree.
\end{abstract}

\section{Introduction}

The cover time of a random walk on a graph, which is the time it
takes the walk to visit every vertex in the graph, is a basic
parameter and has been researched intensively over the last several
decades (see \cite{AF, LPW09, LP} for  background). One often
studied aspect concerns
precise estimates for cover times on specific graphs including 2D
lattices and regular trees.  For the 2D discrete torus, the
asymptotics of the cover time were established by Dembo, Peres,
Rosen and Zeitouni \cite{DPRZ04}. For regular trees, the asymtotics
of the cover time were evaluated by Aldous \cite{Aldous91} and a
tightness result for the cover time after suitable normalization was
demonstrated by Bramson and Zeitouni \cite{BZ09}. It was conjectured
in \cite{BZ09} that the cover time of
2D discrete tori exhibits a similar tightness behavior.

Meanwhile, the supremum of the Gaussian free field (GFF) was also
heavily studied. For squares in the 2D lattice, the first order
asymptotics were evaluated  by
Bolthausen, Deuschel and Giacomin
\cite{BDG01}. Interestingly, both \cite{BDG01} and \cite{DPRZ04} are
based on the study of similar tree structures for the 2D lattice; in
fact, the square of the GFF has the same first order asymptotics as
the cover time after proper normalization. Recently, Ding, Lee, and
Peres \cite{DLP10} demonstrated a useful connection between cover
times and GFFs, by showing that, for any graph, the cover time is
equivalent, up to a universal multiplicative constant, to the
product of the number of edges and the supremum of the GFF. An
important ingredient in \cite{DLP10} is a version of the so-called
Dynkin Isomorphism theorem, which completely characterizes the
distribution of local times (closely related to the cover time)
using GFFs. All these connections seem to suggest that a detailed
study of fluctuations for one model should carry over to the other
with moderate work. A particular motivating example in this
direction is the case of squares in the 2D discrete lattice.
Recently, Bramson and Zeitouni \cite{BZ10} established a tightness
result for the supremum of GFF there (with proper centering, but no
other normalization), and further computed the centering up to an
additive constant.
One could hope that transferring this result to the cover time
problem is now ``purely technical''; an essential part of such a
program would be to verify that the supremum of the GFF correctly
predicts the second order correction for the (rescaled) cover time.
The present paper is a cautionary note in that direction.

We study the cover time on binary trees and obtain the sharp second
order term. Interestingly, we demonstrate that the latter is larger
than the corresponding one for the binary tree GFF.
Our result improves the estimate in \cite{Aldous91}, and complements
the result of \cite{BZ09}. We focus here on binary trees, but it
should be clear from the proof that the method applies to more
general Galton--Watson trees. Recently, Ding obtained the
asymptotics of the cover time on general trees via Gaussian free
fields, together with an
exponential concentration around its mean \cite{Ding11}. We note that the estimates 
in \cite{Ding11} did not yield the correct second order term for the
cover time on the binary tree.

Let $T = (V, E)$ be a binary tree rooted at $\rho$ of height $n$,
and consider a continuous-time random walk $(X_t)$ started at
$\rho$. Let $\tau_{\mathrm{cov}}$ be the first time when the random
walk visited every single vertex in the tree. Our main result is the
following.
\begin{theorem}\label{thm-cover-time}
Consider a random walk on a binary tree $T = (V, E)$ of height $n$,
started at the root $\rho$. Then, with high probability,
\begin{equation}\label{eq-cover-time}
\sqrt{\tfrac{\tau_{\mathrm{cov}}}{|E|}} = \sqrt{2\log 2}\cdot n -
\tfrac{\log n}{\sqrt{2\log 2}} + O((\log\log n)^8)\,.\end{equation}
\end{theorem}
At the cost of a more refined analysis, we believe that the error
term $O((\log \log n)^8)$ can be improved to $O(1)$.

To relate Theorem \ref{thm-cover-time} to the GFF
$\{\eta_v\}_{v\in V}$ on the tree, recall that the latter can be
defined as follows.
Let $\{X_e\}_{e\in E}$ be i.i.d.\ standard Gaussian variables and
set
$$\eta_v = \sum_{e: e\in \rho \leftrightarrow v} X_e,$$
where the sum is over all the edges that belong to the path from
$\rho$ to $v$. By adapting
Bramson's arguments on branching Brownian motion \cite{Bramson78} to
the discrete setup, as in Addario-Berry and Reed \cite{BR09}, one
can show that
\begin{equation}\label{eq-gaussian-free-field} \E
\sup_v \eta_v = \sqrt{2 \log 2}\cdot n - \tfrac{3 \log
n}{2\sqrt{2\log 2}} + O(1)\,.
\end{equation}
(The lower bound in \eqref{eq-gaussian-free-field} follows directly
from \cite[Theorem 3]{BR09}. The upper bound, that involves also the
internal nodes of the tree, requires the use of \cite[Lemma
13]{BR09} and a union bound over the levels.)

Comparing \eqref{eq-cover-time} and \eqref{eq-gaussian-free-field},
we do observe agreement in the first order and a discrepancy in the
second order terms.

Our proof uses ideas from  \cite{Bramson78} and is based on the
study of the local times associated with the random walk.  For any
$v\in V$, we define the local time $L^v_t$ to be the time that the
random walk spends at $v$ up to $t$, with a normalization by the
degree of $v$. More precisely,
$$L^v_t = \frac{1}{d_v} \int_0^t \one_{\{X_s = v\}} ds\,.$$
Define the inverse local time $\tau(t)$ to be the first time when
the local time at the root achieves $t$, by
$$\tau(t) = \inf\{s\geq 0 : L^\rho_s \geq t\}\,.$$
We will let $\tau(t)$ be defined as above throughout the paper. We
also set
\begin{equation}
\label{eq-tpm} t^+ = \big(\sqrt{\log 2} n - \tfrac{\log n}{2
\sqrt{\log 2}} + 100 \log\log n\big)^2 \mbox{ and }t^- =
\big(\sqrt{\log 2}n -\tfrac{\log n}{2\sqrt{\log 2}} - 100 (\log \log
n)^8\big)^2\,.
\end{equation}
The following is the key to the proof of
Theorem~\ref{thm-cover-time}.

\begin{theorem}\label{thm-cover-inverse}
Consider a random walk on a binary tree $T$ of height $n$, started
at the root $\rho$.
Then,
$$\P(\tau(t^-) \leq  \tau_{\mathrm{cov}} \leq \tau(t^+)) = 1+o(1) \,, \mbox{ as } n\to \infty\,.$$
\end{theorem}

In the next two sections, we prove the upper and lower bounds for
the preceding theorem respectively;  we conclude the paper by
deriving Theorem~\ref{thm-cover-time} from
Theorem~\ref{thm-cover-inverse}.\\

\noindent
{\bf Notation and convention:} Throughout, $C,c$ denote generic constants
that may change from line to line, but are independent of $n$. Further,
the
 phrase {\it with high probability} should be understood as the statement
{\it with probability approaching $1$ as $n\to\infty$}.

\section{Upper bound}

We establish an upper bound on the cover time in this section, as
formulated in the next theorem.
\begin{theorem}\label{thm-upper-bound}
With notation as in Theorem \ref{thm-cover-inverse}, we have
$$\P(\tau_{\mathrm{cov}} \leq \tau(t^+)) =
1+o(1)\,, \mbox{ as } n\to\infty\,.$$
\end{theorem}
Theorem \ref{thm-upper-bound}
 is equivalent to
the statement that at time $\tau(t^+)$, all the leaf-nodes have
positive local times, with high probability. To this end, we
consider a leaf-node of local time $0$ with typical and non-typical
profiles, respectively. For the latter, we show its unlikeliness
directly; for the former, we prove it is a rare event by comparing
to the same type of event for Gaussian free field.

\subsection{Unlikeliness for non-typical profile}
As preparation, we prove a large deviation result which will be used
to control the pairwise concentration of local times.

\begin{defn}
    For $r, \lambda >0$, let $N$ be a Poisson
variable with mean $r$ and $Y_i$ be i.i.d.\ exponential variables
with mean $\lambda$.
Then,
 the random variable $Z = \sum_{i=1}^N
Y_i$ is said to follow the distribution $\mathrm{PoiGamma}(r, \lambda)$,
and we write $Z\sim
\mathrm{PoiGamma}(r, \lambda)$.
\end{defn}

\begin{lemma}\label{lem-large-deviation}
For $\alpha,r>0$,  let $Z
\sim
\mathrm{PoiGamma}(r, \lambda)$. Then for $\alpha < \lambda r$,
\begin{equation}\label{eq-lower-tail} \P (Z \leq \lambda r - \alpha)
\leq \exp\big( 2\sqrt{r(r - \alpha/\lambda)} +
\alpha/\lambda-2r\big)\,.
\end{equation}
Furthermore, for all $\alpha
> 0$,
\begin{equation}\label{eq-upper-tail} \P(Z \geq \lambda r + \alpha)\leq
 \exp\big(2\sqrt{r (r+ \alpha/\lambda)} - 2r - \alpha/\lambda\big)\,.
\end{equation}
\end{lemma}
\begin{proof}
As in the definition of the $\mathrm{PoiGamma}(r, \lambda)$
distribution, let $N$ be  Poisson variable with mean $r$ and let $Y$
be an independent exponential variable with mean $\lambda$. For
$\theta>0$, we have
$$\E\mathrm{e}^{-\theta Z/\lambda} = \E (\E \mathrm{e}^{-\theta Y / \lambda})^N
 = \E (1/(1+\theta))^N = \mathrm{e}^{-\frac{\theta r}{1+\theta}}\,.$$
Combined with Markov's inequality, it follows that
$$\P(Z \leq \lambda r - \alpha) = \P(\mathrm{e}^{-\theta Z/\lambda}
\geq \mathrm{e}^{-\theta(\lambda r-\alpha)/\lambda}) \leq
\mathrm{e}^{-\frac{\theta r}{1+\theta}} \cdot
\mathrm{e}^{\theta(r-\alpha/\lambda)} = \exp\big(\tfrac{\theta^2
r}{1+\theta} - \tfrac{\theta \alpha}{\lambda}\big)\,.$$ For $\alpha
< \lambda r$, optimizing the exponent at $\theta = \sqrt{\frac{r}{r
- \alpha/\lambda}} -1$ leads to inequality \eqref{eq-lower-tail}.

To prove \eqref{eq-upper-tail}, consider $0<\theta<1$. We have
$$\E\mathrm{e}^{\theta Z/\lambda} = \E (\E \mathrm{e}^{\theta Y / \lambda})^N =
 \E (1/(1-\theta))^N = \mathrm{e}^{\frac{\theta r}{1-\theta}}\,.$$
Another application of Markov's inequality gives that
$$\P(Z \geq \lambda r + \alpha) \leq \P(\mathrm{e}^{\theta Z/\lambda}
\geq \mathrm{e}^{\theta (\lambda r+\alpha)/\lambda}) =
\exp\big(\tfrac{\theta^2 r}{1- \theta} - \tfrac{\theta
\alpha}{\lambda} \big)\,.$$ Optimizing the exponent at $\theta = 1 -
\sqrt{\tfrac{r}{r+ \alpha/\lambda}}$, we deduce the inequality
\eqref{eq-upper-tail}.
\end{proof}

\noindent{\bf Remark.} The right side of \eqref{eq-lower-tail} can
be bounded by $\mathrm{e}^{-\alpha^2/4\lambda^2 r}$. In this form,
it is closely related to the discrete time bound in
\cite[Lemma 5.2]{KKLV00}.

We have the following immediate and useful corollary.
\begin{cor}\label{cor-large-deviation}
With notation as in Lemma \ref{lem-large-deviation}, we have for any
$\beta>0$,
\begin{equation}\label{eq-lower-tail-bis} \P (\sqrt{Z}
    \leq (1-\beta)\sqrt{\lambda r})\leq \mathrm{e}^{-r\beta^2}\,,
\end{equation}
and
\begin{equation}\label{eq-upper-tail-bis}
\P(\sqrt{Z} \geq (1+\beta)\sqrt{\lambda r}) \leq
\mathrm{e}^{-r\beta^2}\,.
\end{equation}
\end{cor}

For $k\in \N$, we denote by $V_k\subseteq V$ the set of vertices in
$k$-th level of the tree. We next show that it is unlikely to have a
too small local time for a vertex in intermediate levels.
\begin{lemma}\label{lem-local-curve}
With notation as in Theorem \ref{thm-cover-inverse}, define
\begin{equation}
\label{def-A}
A = \cup_{k=1}^{n-\log^2 n}\cup_{u\in V_k} \{L^u_{\tau(t^+)} \leq
((1-k/n)\sqrt{t^+} - 3\log n)^2\}\,.\end{equation} Then, $\P(A) =
o(1)$ as $n\to \infty$.
\end{lemma}
\begin{proof}
Throughout the proof, we write $t$ for $t^+$. Consider $u\in V_k$
such that $k\leq n-\log^2 n$. It is clear that $L^u_{\tau(t)}$ has
the distribution $\mathrm{PoiGamma}(t/k, k)$.
Applying \eqref{eq-lower-tail}, we obtain that
$$\P\left(Z \leq ((1-k/n)\sqrt{t} - 3\log n)^2\right) \leq \exp\big(-\tfrac{1}{k}(\sqrt{t}k/n + 3\log n)^2\big) \leq 2^{-k}n^{-2}\,.$$
Now a simple union bound gives that
\begin{equation*}
 \P(A)\leq  \sum_{k=1}^{n - \log^2 n} 2^k 2^{-k} n^{-2} \leq  1/n = o(1)\,. \qedhere
\end{equation*}
\end{proof}



For $v\in V$ and $1\leq k< n$, let $v_k\in V_k$ be the ancestor of
$v$ in the $k$-th level. Define
\begin{equation}\label{eq-def-gamma}
\gamma(k) = \min\{\sqrt{k}\log k, \, \sqrt{n-k} \log (n-k)\} +2
\,.\end{equation}

\begin{lemma}
    With notation as in Theorem
\ref{thm-cover-inverse},
define
\begin{equation}\label{eq-def-B}
B= \left\{\exists v\in V_n, \exists k<  n- \log^2 n:
L^{v}_{\tau(t^+)} = 0, \left|\sqrt{L^{v_{k+1}}_{\tau(t^+)}} -
\sqrt{L^{v_{k}}_{\tau(t^+)}}\right| \geq
\tfrac{\sqrt{L^{v_k}_{\tau(t^+)}}}{\gamma(k)}\right\} \cap
A^c\,.\end{equation} Then, $\P(B) = o(1)$  as  $n \to \infty$.
\end{lemma}

\begin{proof}
    We continue to write $t=t^+$.
Consider $v\in V_n$ and $k< n- \log^2 n$. Note that conditioned on
$L^{v_k}_{\tau(t)}$, the collection of random variables
$\{L^{v_{k+j}}_{\tau(t)}\}_{j\geq 0}$ possess the same law as
$\{L^{v_{k+j}}_{\tau^{v_k}(L^{v_k}_{\tau(t)})}\}_{j\geq 0}$ (this is
an instance of the second Ray-Knight theorem in this context).
Abusing notation, this implies in particular that conditioned on
$\{L^{v_k}_{\tau(t)}=x\}$, $L^{v_{k+1}}_{\tau(t)}$ has distribution
$\mathrm{PoiGamma}(x^2,1)$. (We will employ such an abuse of
notation repeatedly throughout the paper.) Fixing $x \geq
(1-k/n)\sqrt{t} - 3\log n$,
an application of Corollary~\ref{cor-large-deviation} gives for $
j\geq 1$,
$$\P\left(j \tfrac{x}{\gamma(k)}\leq \left|\sqrt{L^{v_{k+1}}_{\tau(t)}} -
x\right| \leq (j+1)\tfrac{x}{\gamma(k)}, L^v_{\tau(t)} = 0 \Big|
L^{v_k}_{\tau(t)} = x^2\right) \leq 2 \cdot \mathrm{e}^{- j^2
x^2/(\gamma(k))^2} \cdot \mathrm{e}^{-\frac{x^2(1 -
(j+1)/\gamma(k))^2}{n-k}}\,.
$$
Note that the right hand side in the above decays geometrically with
$j$. Thus, summing over $j$, we obtain that
$$\P\left(\left|\sqrt{L^{v_{k+1}}_{\tau(t)}} -
x\right| \geq \tfrac{x}{\gamma(k)}, L^v_{\tau(t)} = 0 \Big|
L^{v_k}_{\tau(t)} = x^2\right) \leq 4 \mathrm{e}^{-\frac{x^2}{n-k}}
\mathrm{e}^{\frac{4x^2}{(n-k) \gamma(k)}} \mathrm{e}^{-\frac{x^2}{
(\gamma(k))^2}} \leq 4 \mathrm{e}^{-\frac{x^2}{n-k}}
 \mathrm{e}^{-\frac{x^2}{2
(\gamma(k))^2}}\,.$$ Noting that $\P(L^v_{\tau(t)} = 0 \mid
L^{v_k}_{\tau(t)} = x^2) = \mathrm{e}^{-\frac{x^2}{n-k}}$, we obtain
that
$$\P\left(\left|\sqrt{L^{v_{k+1}}_{\tau(t)}} -
x\right| \geq \tfrac{x}{\gamma(k)} \Big| L^v_{\tau(t)} = 0 ,
L^{v_k}_{\tau(t)} = x^2\right) \leq 2\mathrm{e}^{-\frac{x^2}{2
(\gamma(k))^2}} \leq \mathrm{e}^{-\log^{3/2} n}\,,$$ where the last
nequality follows from the fact that $x \geq (1-k/n)\sqrt{t} - 3\log
n)$ and $k\leq n - \log^2 n$. Therefore,
$$\P\left(A^c, L^v_{\tau(t)} = 0, \left|\sqrt{L^{v_{k+1}}_{\tau(t)}} -
\sqrt{L^{v_{k}}_{\tau(t)}}\right| \geq
\tfrac{\sqrt{L^{v_k}_{\tau(t)}}}{\gamma(k)}\right) \leq
\P(L^v_{\tau(t)} = 0)\cdot \mathrm{e}^{-\log^{3/2} n} =
\mathrm{e}^{-t/n}\mathrm{e}^{-\log^{3/2} n} \leq
\frac{2^{-n}}{n^2}\,.$$ At this point, a simple union bound
completes the proof.
\end{proof}

\subsection{Unlikeliness for typical profile}

We next compare the density of local times and Gaussian variables.
This comparison of density is of significance for the proof of both
upper and lower bounds.

\begin{lemma}\label{lem-density-ratio}
    For $\ell>0$, let $Z \sim \mathrm{PoiGamma}(\ell^2,1)$ and let
$f(\cdot)$ denote the density function of $\sqrt{Z}$ on $R_+$, with
$f(0) = \P(Z = 0)$. Denote by $W$ a standard Gaussian variable, and
denote by $g(\cdot)$ the density function of $W/\sqrt{2}$. Then, for
any $w$ such that $|w|\leq \ell/2$, we have
$$f(\ell + w) = \big(1 - \tfrac{w}{2\ell} + O\big(\tfrac{w^2+1}{\ell^2}\big)\big) \cdot g(w)\,.$$
\end{lemma}
\begin{proof}
Write $y = \ell + w$, and let $h(\cdot)$ be the density function of
$Z$. Then for $z>0$, we have
$$h(z) = \sum_{k=1}^\infty \mathrm{e}^{-\ell^2} \frac{\ell^{2k}}{k!} \mathrm{e}^{-z}\frac{z^{k-1}}{(k-1)!}\,.$$
Applying a change of variables, we obtain that
$$f(y) = 2y\sum_{k=1}^\infty \mathrm{e}^{-\ell^2} \frac{\ell^{2k}}{k!} \mathrm{e}^{-y^2}\frac{y^{2(k-1)}}{(k-1)!}
= 2 \ell\mathrm{e}^{-(\ell^2+y^2)}\sum_{k=0}^\infty
\frac{(y\ell)^{2k+1}}{k! (k+1)!} = 2\ell \mathrm{e}^{-(\ell^2+y^2)}
I_1(2y\ell)\,,$$ where $I_1(x)$ is a modified Bessel function
defined by
$$I_1(x) \deq \sum_{k=0}^\infty \frac{(x/2)^{2k+1}}{k! (k+1)!}\,.$$
For the modified Bessel function $I_1(x)$, the following expansion
is known when $|x|$ is large (see \cite{AS65}):
$$I_1(x) = \frac{\mathrm{e}^x}{\sqrt{2\pi x}} \big(1 - \tfrac{3}{8x} + O\big(\tfrac{1}{x^2}\big)\big)\,.$$
Plugging into the preceding expansion, we get that
$$f(y) = 2\ell \mathrm{e}^{-(\ell^2+y^2)} \frac{\mathrm{e}^{2y\ell}}{\sqrt{2\pi 2y\ell}} \Big(1 - \tfrac{3}{8 \cdot 2y\ell} + O\big(\tfrac{1}{y^2\ell^2}\big)\Big)
= \frac{\mathrm{e}^{-w^2}}{\sqrt{\pi}} \big(1- \tfrac{w}{2\ell} +
O\big(\tfrac{w^2+1}{\ell^2}\big)\big)\,.$$ Combined with the fact
that $g(w) = \frac{1}{\sqrt{\pi}}\mathrm{e}^{-w^2}$, the desired
estimate follows immediately.
\end{proof}

We single out the next calculation, which will be used repeatedly.
\begin{claim}\label{claim-technical}
Consider $z_i, \ell_i\in \R$ with $\ell_{i+1} = \ell_i + z_i$ for
$i=0, \ldots, m-1$ such that $|z_i| \leq \ell_i/2$ for all $i$.
Assume that $\sum_i \frac{z_i^2 + 1}{\ell_{i-1}^2} = O(1)$. Then,
$$\mbox{$\prod_{i=1}^m$} \big(1 - \tfrac{z_i}{2\ell_{i-1}} + O\big(\tfrac{z_i^2 + 1}{\ell_{i-1}^2}\big)\big) = \Theta(1) \cdot \frac{\sqrt{\ell_0}}{\sqrt{\ell_m}}\,.$$
\end{claim}
\begin{proof}
On one hand, note that
$$\frac{\ell_m}{\ell_0} = \mbox{$\prod_{i=1}^{m}$}\frac{\ell_{i}}{\ell_{i-1}} = \mbox{$\prod_{i=1}^{m}$}\big(1 + \tfrac{z_i}{\ell_{i-1}}\big) = \exp\big(\mbox{$\sum_{i=1}^{m}$}\tfrac{z_i}{\ell_{i-1}} + O\big(\mbox{$\sum_{i=1}^{m}$}\tfrac{z_i^2}{\ell_{i-1}^2}\big)\big)=\exp\big(\mbox{$\sum_{i=1}^{m}$}\tfrac{z_i}{\ell_{i-1}} + O(1)\big)\,.$$
On the other hand, we have
\begin{align*}\mbox{$\prod_{i=1}^{m}$}\big(1 - \tfrac{z_i}{2\ell_{i-1}} +
O\big(\tfrac{z_i^2 + 1}{\ell_{i-1}^2}\big)\big) &=
\exp\big(-\mbox{$\sum_{i=1}^{m}$}\tfrac{z_i}{2\ell_{i-1}} +
O\big(\mbox{$\sum_{i=1}^{m}$}\tfrac{z_i^2 + 1}{\ell_{i-1}^2}\big)\big)\\
& =\exp\big(-\mbox{$\sum_{i=1}^{m}$}\tfrac{z_i}{2\ell_{i-1}} +
O(1)\big) = \sqrt{\tfrac{\ell_0}{\ell_m}}\exp(O(1))\,.
\end{align*}
Combining these estimates completes the proof.
\end{proof}

We next demonstrate that it is unlikely to have a leaf-node of local
time 0, even with a typical profile for local times along the path
from $\rho$ to the leaf.
\begin{lemma}
     With notation as in Theorem
\ref{thm-cover-inverse}
and $A,B$ as in \eqref{def-A} and \eqref{eq-def-B},
%
define
\begin{equation}D_v = \{L^v_{\tau(t^+)} = 0\} \setminus (A \cup B)\,,
\mbox{ for } v\in V_n\,.\end{equation} Then, $\P(D_v) = o(2^{-n})$.
\end{lemma}
\begin{proof}
    Again, we write $t=t^+$.
Write $n' = n - \log^2 n$. Let $\Omega\subseteq \R^{n'}$ be such
that for ${z_1, \ldots, z_{n'}} \in \Omega$, we have
$$\cap_{k=1}^{n'}\left\{
\sqrt{L^{v_{k}}_{\tau(t)}} -
\sqrt{L^{v_{k-1}}_{\tau(t)}}=z_k\right\}
\subseteq D_v\,.$$ Let $\alpha(\cdot)$ and $\beta(\cdot)$ be density
functions for $(\sqrt{L^{v_{k}}_{\tau(t)}} -
\sqrt{L^{v_{k-1}}_{\tau(t)}})_{1\leq k\leq n'}$ and
$(\eta_{v_k}/\sqrt{2} - \eta_{v_{k-1}}/\sqrt{2})_{1\leq k\leq n'}$,
respectively. Denote by $\ell_k = \sqrt{t} + \sum_{i=1}^k z_i$. Note
that for $(z_1, \ldots, z_{n'})\in \Omega$, we have
$$\mbox{$\sum_{i=1}^{n'}$} \tfrac{1 + z_i^2}{\ell_{i-1}^2} = O(1)
\mbox{$\sum_{i=1}^{n'}$} \big(\tfrac{1}{(n-i)^2} +
\tfrac{1}{(\gamma(i))^2}\big) = O(1)\,.$$ Applying
Lemma~\ref{lem-density-ratio} and Claim~\ref{claim-technical}, we
obtain that
\begin{align*}\tfrac{\alpha(z_1, \ldots, z_{n'})}{\beta(z_1, \ldots, z_{n'})}
= &\mbox{$\prod_{i=1}^{n'}$}\big(1 - \tfrac{z_i}{2\ell_{i-1}} +
O\big(\tfrac{z_i^2 + 1}{\ell_{i-1}^2}\big)\big) = \Theta(1)
\tfrac{\sqrt{n}}{\log n}\,.\end{align*} Therefore, we obtain that
\begin{align}\label{eq-P-D}\P(D_v) = \int_{\Omega} \alpha(z_1,
\ldots, z_{n'}) \P(L^v_{\tau(t)} &= 0 \mid
\sqrt{L^{v_{n'}}_{\tau(t)}} = \ell_{n'}) dz \leq O(1)
\frac{\sqrt{n}}{\log n}\int_\Omega \beta(z_1, \ldots, z_{n'})
\mathrm{e}^{-\frac{\ell_{n'}^2}{n-n'}}dz\,.\end{align} Write $s =-
(n'/n)\sqrt{t} - 3\log n$. Let $\beta(x) = \int_{\{\ell_{n'} = x\}}
\beta(z_1, \ldots, z_{n'})dz$ for $x\geq s$. Note that
\begin{equation}\label{eq-beta}
\beta(x) = \frac{1}{\sqrt{\pi n'}}\mathrm{e}^{-\frac{x^2}{n'}}
\P(\eta_{v_k}/\sqrt{2} \geq -(k/n)\sqrt{t} - 3 \log n \mbox{ for }
1\leq k\leq n'\mid \eta_{v_{n'}}/\sqrt{2} = x)\,.\end{equation}
Conditioning on $\eta_{v_{n'}}/\sqrt{2} = x$, we have
$$\{(\eta_{v_k} /\sqrt{2})_{1\leq k\leq n'} \mid \eta_{v_{n'}}/\sqrt{2} = x\} \stackrel{law}{=} \{(W_k/\sqrt{2} + (k/n')x)_{1\leq k\leq n'} \}\,,$$
where $(W_r)_{0\leq r\leq n'}$,
is a Brownian Bridge of length $n'$,
i.e., a Brownian motion conditioned
on hitting $0$ at both time $0$ and $n'$. It is well-known that the
maximum of a Brownian bridge $(W_r)$ on $[0, q]$ follows the
Rayleigh
distribution (see, e.g., \cite{SW86}), i.e.,
\begin{equation}\label{eq-reighlay}
\P(\mbox{$\max_{0\leq r\leq q}$} W_r \geq \lambda) =
\mathrm{e}^{-\frac{2 \lambda^2}{q}}\,, \mbox{ for all } \lambda \geq
0\,.
\end{equation}
 Therefore, we obtain that
\begin{align*}\P&(\eta_{v_k}/\sqrt{2} \geq -(k/n)\sqrt{t} - 3 \log n \mbox{
for } 1\leq k\leq n'\mid \eta_{v_{n'}}/\sqrt{2} = x) \leq
\P(\mbox{$\min_{r\leq n'}$}
W_r/\sqrt{2} \geq -3 \log n - (x-s)) \\
&= \P( \mbox{$\max_{r\leq n'}$} W_r \leq \sqrt{2}(3\log n+(x-s)))
\leq \tfrac{4(3\log n+(x-s))^2}{n'}\,.\end{align*}  Plugging the
above estimate into \eqref{eq-beta}, we obtain that
$$\beta(x)\leq \frac{4(3\log
n+(x-s))^2}{(n')^{3/2}}\mathrm{e}^{-\frac{x^2}{n'}}\,.$$
Together
with \eqref{eq-P-D}, we obtain that
$$\P(D_v) \leq O(1)
\int_s^\infty \frac{(\log n+(x-s))^2}{n' \log n}
\mathrm{e}^{-\frac{x^2}{n'}} \mathrm{e}^{-\frac{(\sqrt{t} +
x)^2}{n-n'}}dx \leq O(1)  \int_{-\infty}^\infty \frac{(\log
n+(x-s))^2}{n' \log n} \mathrm{e}^{-\frac{x^2}{n'}}
\mathrm{e}^{-\frac{(\sqrt{t} + x)^2}{n-n'}}dx\,.$$ Using
the change of
variables $y = x + \frac{\sqrt{t} n'}{n}$, we obtain that
\begin{align*}
\P(D_v) \leq O(1) \frac{\mathrm{e}^{-\frac{t}{n}}}{n \log n} \cdot
\int_{-\infty}^\infty(4\log n+y)^2 \mathrm{e}^{-(\frac{1}{n'} +
\frac{1}{n-n'})y^2} dy = 2^{-n} \cdot o(\log^{-6} n)\,,
\end{align*}
where we used the fact that $n' = n - \log^2 n$, completing the
proof.
\end{proof}

\noindent {\bf Proof of Theorem~\ref{thm-upper-bound}.} The proof
now follows trivially. Since $\sum_{v\in V_n}\P(D_v) = 2^n \cdot
2^{-n} o(1) = o(1)$ as well as $\P(A) = o(1)$ and $\P(B) = o(1)$, we
see that with high probability, every leaf-node has positive local
time by $\tau(t)$, implying the desired upper bound on cover time.

\section{Lower bound}
This section is devoted to the proof of the following lower bound on the
cover time for a binary tree $T$.
\begin{theorem}\label{thm-lower-bound}
      With notation as in Theorem
\ref{thm-cover-inverse},
%
$$\P(\tau_{\mathrm{cov}} \geq \tau(t^-)) = 1+ o(1)\,, \mbox { as } n \to \infty\,.$$
\end{theorem}
The proof consists of an analysis for exceptionally large values in
the Gaussian free field and a comparison argument based on
Lemma~\ref{lem-density-ratio}.

\subsection{Exceptional points for Gaussian free field}
We first study the Gaussian free field $\{\eta_v\}_{v\in V}$ on the
tree $T$ of height $n$, with $\eta_\rho = 0$. For $1\leq k<n$, let
$\psi(k) = \frac{\log (k \wedge (n-k))}{2\sqrt{\log 2}}$.  Denote by
\begin{equation}\label{eq-def-a} a_k = (k/n)\big(\sqrt{\log 2} n -
\tfrac{\log n}{2\sqrt{\log 2}}\big) - \psi(k) + 2\,, \mbox{ for }1
\leq k < n\,, \mbox{ and } a_n = \sqrt{\log 2} n - \tfrac{\log
n}{2\sqrt{\log 2}}\,.\end{equation} Consider $\Delta = a_n + \log^4
n$. Recall the definition of $\gamma(k)$ in \eqref{eq-def-gamma}.
For $v\in V_{n}$, define
\begin{align}
E_v &= \{\eta_{v_k}/\sqrt{2} \leq a_k , \mbox{ for all } 1\leq k <
n,  a_n \leq \eta_{v}/\sqrt{2} \leq a_{n} + 1
\}\,,\label{eq-def-E-v}\\
F_v &= \{E_v,  \exists k\leq n: |\eta_{v_k} - \eta_{v_{k-1}}| \geq
|\Delta-\eta_{v_{k-1}}/\sqrt{2}|/\gamma(k)\}\,.\label{eq-def-F-v}
\end{align}

We start with a lower bound on the probability for event $E_v$.

\begin{lemma}\label{lem-E-v}
There exists a constant $c>0$ such that for all $v\in V_n$, we have
$$\P(E_v) \geq \frac{c}{\sqrt{n}} 2^{-n}\,.$$
\end{lemma}
\begin{proof}
It is clear that
$$\P(E_v) \geq \P(a_n\leq \eta_v/\sqrt{2} \leq a_n+1)
\min_{a_n\leq x\leq a_n+1}\P(E_v \mid \eta_v = \sqrt{2}x) \geq
\frac{\sqrt{n}}{5\cdot 2^n}\cdot \min_{a_n\leq x\leq a_n+1}\P(E_v
\mid \eta_v = \sqrt{2}x)\,,$$ where the second inequality follows
from a bound on the Gaussian density. Denote by $(W_t)_{0\leq t\leq
n}$ a Brownian bridge. We note that
$$\left(\{\eta_{v_\ell}: 0 \leq \ell \leq n\} \mid \eta_v = \sqrt{2}x \right)\stackrel{law}{=} \left\{W_\ell + \tfrac{\ell}{n}\sqrt{2}x : 0\leq \ell \leq n\right\}\,.$$
This implies that, for $x\geq a_n$,
$$\P(E_v \mid \eta_v = \sqrt{2}
x) \geq \P(W_\ell \leq 1 - \sqrt{2} \psi(\ell)\; \mbox{ for } 0\leq
\ell\leq n)\,.$$
By \cite[Proposition 2']{Bramson78},
we have
that $\P(W_\ell \leq 1 - \sqrt{2} \psi(\ell) \;\mbox{for } 0\leq
\ell \leq n) \geq c/n$ for a constant $c>0$. Altogether, we obtain
that \begin{equation*}\P(E_v) \geq \frac{c}{5\sqrt{n}} 2^{-n}\,.
\qedhere\end{equation*}
\end{proof}

We now show that the event $F_v$ is extremely rare.

\begin{lemma}\label{lem-F-v}
For any $v\in V_n$, we have
$$\P( F_v) = 2^{-n}o(1/n), \mbox{ as } n\to \infty\,.$$
\end{lemma}
\begin{proof}
Take $v\in V_{n}$. It is clear that
\begin{align*}
\P(F_v) \leq  \P(a_{n} \leq \eta_v/\sqrt{2} \leq a_{n} + 1) J \leq
2^{-n} n J\,,
\end{align*}
where
$$J = \max_{\stackrel{a_{n} \leq x \leq
a_{n} + 1}{ y\leq a_{k-1}}}\sum_{k=1}^{n} \P(|\eta_{v_k} -
\eta_{v_{k-1}}| \geq |\Delta-\eta_{v_{k-1}}/\sqrt{2}|/\gamma(k) \mid
\eta_v = \sqrt{2}x,  \eta_{v_{k-1}} = \sqrt{2} y)\,.$$ Conditioning
on $\eta_v =\sqrt{2} x,  \eta_{v_{k-1}} = \sqrt{2}y$, we have
$\eta_{v_k} - \eta_{v_{k-1}}$ distributed as a Gaussian variable
with mean $\frac{\sqrt{2}}{n-k + 1} (x-y)$ and variance
$\frac{n-k}{n-k+1}$. For $x,y$ that is under consideration, we have
$\frac{\sqrt{2}}{n-k + 1} (x-y) = o(\frac{\Delta - y}{\gamma(k)})$.
Therefore, we obtain that
$$\P(|\eta_{v_k} -
\eta_{v_{k-1}}| \geq |\Delta -\eta_{v_{k-1}}/\sqrt{2}|/\gamma(k)
\mid \eta_v = \sqrt{2}x,  \eta_{v_{k-1}} = \sqrt{2} y)\leq
\mathrm{e}^{-\frac{(\Delta - y)^2}{4(\gamma(k))^2}} \leq
\mathrm{e}^{-\log^2 n}\,,$$ for large enough $n$. This implies that
$J \leq n \mathrm{e}^{-\log^2 n}$, and thus $\P(F_v) =
2^{-n}o(1/n)$.
\end{proof}

We next study the correlation for events $E_u$ and $E_v$. For $u,
v\in V$, denote by $u\wedge v$ the least common ancestor of $u$ and
$v$.

\begin{lemma}\label{lem-E-u-v}
Consider $u, v\in V_{n}$ and assume that $u\wedge v\in V_k$. Then,
$$\P(E_u \cap E_v)\leq \P(E_u) \frac{20 \log^2 n}{\sqrt{n-k}\cdot ((n-k) \wedge k) } 2^{-(n-k)}\,.$$
\end{lemma}
\begin{proof}
Denote by $w = u \wedge v$, and let $f(\cdot)$ be the density
function of $\eta_w/\sqrt{2}$. For $i<j$, write $E_v^{i, j} =
\{\eta_{v_\ell}/\sqrt{2} \leq a_\ell , \mbox{ for all } i\leq \ell <
j\}$. Then,
\begin{align}
\P(E_u \cap E_v) &= \P(E_u) \P(E_v \mid E_u) \leq \P(E_u) \max_{x
\leq a_k} \P(E_v^{k, n}, a_n \leq \eta_v \leq a_n + 1  \mid \eta_w/\sqrt{2} = x)\nonumber\\
&\leq \P(E_u) \max_{x\leq a_k}\int_{a_n}^{a_n+1}
\frac{1}{\sqrt{n-k}} \mathrm{e}^{-\frac{(y-x)^2}{n-k}} \P(E_v^{k, n}
\mid \eta_w/\sqrt{2} = x, \eta_v/\sqrt{2} = y) dy \,.
\label{eq-E-u-v-temp}
\end{align}
%
For $x \leq a_k$ and $a_n \leq y a_n +1$, we first analyze the
probability $\P(E_v^{k, n} \mid \eta_w/\sqrt{2} = x, \eta_v/\sqrt{2}
= y)$. Let $(W_s)_{0\leq s \leq n-k}$ be a Brownian bridge. It is
clear that
$$\left(\left\{\eta_{v_\ell} /\sqrt{2}: k\leq \ell \leq n\right\} \mid \eta_w/\sqrt{2} = x, \eta_v/\sqrt{2} = y\right) \stackrel{law}{=} \left\{W_{\ell - k}/\sqrt{2} + \tfrac{\ell - k}{n-k} y + \tfrac{n- \ell}{n-k}x: k\leq \ell \leq n\right\}\,.$$
Combined with \eqref{eq-reighlay}, it follows that
\begin{equation}\label{eq-E-v-k-n}
\P(E_v^{k, n} \mid \eta_w/\sqrt{2} = x, \eta_v/\sqrt{2} = y) \leq
\P(\max_s W_s \leq 2(\log n + (a_k-x))) \leq \frac{4(\log n +
(a_k-x))^2}{n-k}\,.\end{equation} By a
straightforward calculation, we
have that \begin{align*}\mathrm{e}^{-\frac{(y-x)^2}{n-k}} \leq
\mathrm{e}^{-\frac{(a_n - a_k)^2}{n-k}}
\mathrm{e}^{-\frac{2(a_n-a_k)(a_k - x)}{n-k}} &\leq 2^{-(n-k)}
\mathrm{e}^{\frac{n-k}{n}\log n - \log ((n-k) \wedge k)}
\mathrm{e}^{2\sqrt{\log 2}(x-a_k)} \\
&\leq 2^{-(n-k)}\frac{n-k}{(n-k)\wedge k} \mathrm{e}^{2\sqrt{\log
2}(x-a_k)}\,.\end{align*} Combined with \eqref{eq-E-v-k-n}, it
follows that
\begin{align*}\int_{a_n}^{a_n + 1} \frac{1}{\sqrt{n-k}}
    \mathrm{e}^{-\frac{(y-x)^2}{n-k}} \P(E_v^{k, n} \mid \eta_w = \sqrt{2}
    x,
    \eta_v = \sqrt{2} y) dy &\leq 2^{-(n-k)}\frac{4(\log n +
(a_k-x))^2}{\sqrt{n-k} \cdot ((n-k) \wedge k)} \mathrm{e}^{2\sqrt{\log 2}(x -a_k)}\\
& \leq 2^{-(n-k)}\frac{20 \log^2 n}{\sqrt{n-k} \cdot ((n-k) \wedge
k)}\,.\end{align*}
 Together
with \eqref{eq-E-u-v-temp}, we deduce that
\begin{equation*}\P(E_u \cap E_v)\leq \P(E_u) \frac{20 \log^2 n}{\sqrt{n-k} \cdot ((n-k) \wedge k)}
2^{-(n-k)}\,.\qedhere\end{equation*}
\end{proof}

\subsection{Lower bound for cover times}

We now turn to study the cover time. The key estimate lies in the
following proposition.

\begin{prop}\label{prop-lower-bound}
With notation as in Theorem \ref{thm-cover-inverse}, let
$s = (\sqrt{\log 2} n
- \frac{\log n}{2\sqrt{\log 2}} + \log^4 n)^2$. Then there exists a
constant $c>0$ such that
$$\P(\mbox{$\min_{v\in V_n}$} L^v_{\tau(s)} \leq \log^8 n) \geq \frac{c}{\log^5 n}\,.$$
\end{prop}
\begin{proof}
Let $Z_v = \sqrt{L^v_{\tau(s)}}$ for $v\in V$. Let $a_k$ be defined
as in \eqref{eq-def-a}. For $v\in V_n$, define
\begin{align}
\tilde{E}_v &= \{ \sqrt{s} - Z_{v_k} \leq a_k , \mbox{ for all }
1\leq k < n, a_n \leq \sqrt{s} - Z_v \leq a_{n} + 1
\}\,,\\
\tilde{F}_v &= \{\tilde{E}_v,  \exists k\leq n: |Z_{v_k} -
Z_{v_{k-1}}| \geq Z_{v_{k-1}}/\gamma(k)\}\,.
\end{align}
Define $\Omega_v \subseteq \R^n$ such that for any $(z_1, \ldots,
z_n) \in \Omega_v$
$$\{Z_{v_{k}} - Z_{v_{k-1}} = z_k \mbox{ for all } 1\leq k\leq n\} \subseteq \tilde{E}_v \setminus \tilde{F}_v\,.$$
It is clear from \eqref{eq-def-E-v} and \eqref{eq-def-F-v} that
$$\{\eta_{v_{k-1}}/\sqrt{2} - \eta_{v_{k}}/\sqrt{2} = z_k \mbox{ for all } 1\leq k\leq n\} \subseteq E_v \setminus F_v\,.$$
Let $\alpha_v(\cdot), \beta_v(\cdot)$ be density functions over
$\Omega$ for $(Z_{v_k} - Z_{v_{k-1}})_{1\leq k\leq n}$ and
$(\eta_{v_{k-1}}/\sqrt{2} - \eta_{v_{k}}/\sqrt{2})_{1\leq k\leq n}$,
respectively. Consider $(z_1, \ldots, z_n) \in \Omega_v$. By
Lemma~\ref{lem-density-ratio}, we have
\begin{equation}\label{eq-alpha-beta}\alpha_v(z_1, \ldots, z_n) =
\beta_v(z_1, \ldots, z_n) \mbox{$\prod_{k=1}^n $} \big(1 -
\tfrac{z_k}{2\ell_{k-1}} + O(\tfrac{z_k^2 +
1}{\ell_{k-1}})\big)\,,\end{equation} where $\ell_k = \sqrt{s} -
\sum_{i=1}^{k} z_k$. Since $(z_1, \ldots, z_n) \in \Omega_v$, we
have that
$$\sum_{k=1}^n \tfrac{z_k^2+1}{\ell_{k-1}^2} \leq \sum_{k=1}^n \big(\tfrac{1}{(\gamma(k))^2} + \tfrac{4}{(n - k)^2+ \log^4 n}\big) = O(1)\,.$$
Applying Claim~\ref{claim-technical}, we obtain that
$$\alpha_v(z_1, \ldots, z_n) = \Theta(1)\frac{\sqrt{n}}{\log^2 n}\beta_v(z_1, \ldots, z_n)\,.$$
Integrating over both sides and recalling Lemmas~\ref{lem-F-v}
and \ref{lem-E-v}, we obtain that \begin{equation}\label{eq-tilde-E-v}
\P(\tilde{E}_v \setminus \tilde{F_v}) =
\Theta(1)\frac{\sqrt{n}}{\log^2 n} \P(E_v \setminus F_v) =
\Theta(1)\frac{c_1 \sqrt{n}}{2 \log^2 n} \P(E_v) \geq \Theta(1)
\cdot \frac{1}{2 \log^2 n}2^{-n}\,.\end{equation}

\medskip

We next analyze the correlation of $\tilde{E}_v\setminus
\tilde{F}_v$ and $\tilde{E}_u\setminus \tilde{F}_u$. Consider $u,
v\in V_n$ and assume that $u \wedge v\in V_k$. We write
\begin{align*}\underline{Z}& = (Z_{v_1} - Z_{v_0}, \ldots, Z_{v_n} - Z_{v_{n-1}},
Z_{u_{k+1}} - Z_{u_{k}}, \ldots, Z_{u_{n}} - Z_{u_{n-1}})\,,\\
\underline{\eta}& = \frac{1}{\sqrt{2}}(\eta_{v_0} - \eta_{v_1},
\ldots, \eta_{v_{n-1}} - \eta_{v_{n}}, \eta_{u_{k}} -
\eta_{u_{k+1}}, \ldots, \eta_{u_{n-1}} - \eta_{u_{n}})\,.
\end{align*}
Define $\Omega_{u, v} \subseteq \R^{2n-k}$ such that for all
$\underline{z} = (z_{v,1}, \ldots, z_{v,n}, z_{u,k+1}, \ldots, z_{u,
n}) \in \Omega_{u, v}$,
$$\{\underline{Z} = \underline{z}\} \subseteq (\tilde{E}_v \setminus \tilde{F}_v) \cap (\tilde{E}_u \setminus \tilde{F}_u)\,.$$
It is then clear that $\{\underline{\eta} = \underline{z}\}
\subseteq (E_v\setminus F_v) \cap (E_u \setminus F_u)$. Let
$\alpha_{u, v}(\cdot)$ and $\beta_{u, v}(\cdot)$ be density
functions for $\underline{Z}$ and $\underline{\eta}$, respectively.
Let $z_{u, i} = z_{v, i}$ for all $1\leq i\leq k$. For $w\in \{u,
v\}$, write $\ell_{w, j} = \sqrt{t} - \sum_{i=1}^j z_{w, j}$. By
Lemma~\ref{lem-density-ratio}, we get that that
$$\alpha_{u, v}(\underline{z}) = \beta_{u, v}(\underline{z})\cdot
\prod_{j=1}^n \big(1 - \tfrac{z_{v, j}}{2\ell_{v, j-1}} +
O\big(\tfrac {z_{v, j}^2+1}{\ell_{v, j-1}^2}\big)\big) \cdot
\prod_{j=k}^n\big( 1 - \tfrac{z_{u,j}}{2\ell_{u,j-1}} +
O\big(\tfrac{z_{u, j}^2 + 1}{\ell_{u, j-1}^2}\big)\big)\,.$$
Applying Claim~\ref{claim-technical} again, we obtain that
$$\alpha(u, v)(\underline{z}) = O(1)\frac{ \sqrt{n(n-k)}}{\log^2 n}
\beta_{u, v}(\underline{z})\,.$$ Integrating over both sides and
applying Lemma~\ref{lem-E-u-v}, we get that
$$\P(\tilde{E}_v \setminus \tilde{F}_v) \cap (\tilde{E}_u
\setminus \tilde{F}_u)  = O(1) \frac{\sqrt{n}}{(n-k) \wedge k}
\P(E_v) 2^{-(n-k)}\,.$$ This implies that for a constant $C>0$
$$\E\big(\mbox{$\sum_{w\in V_n}$} \one_{\tilde{E}_w\setminus \tilde{F}_w}\big)^2 \leq C \sqrt{n} 2^n \P(E_v)  \sum_{j=1}^n \sum_{w: w\wedge v \in V_j} \frac{2^{-(n-j)}}{(n-j) \wedge j} \leq C \sqrt{n} 2^n \P(E_v) \cdot 4\log n \,.$$
At this point, an application of the second moment method together with
\eqref{eq-tilde-E-v} gives that
$$\P(\exists w\in V_n: \tilde{E}_w\setminus \tilde{F}_w) \geq \frac{\big(\E\mbox{$\sum_{w\in V_n}$} \one_{\tilde{E}_w\setminus \tilde{F}_w}\big)^2}{\E\big(\mbox{$\sum_{w\in V_n}$} \one_{\tilde{E}_w\setminus \tilde{F}_w}\big)^2} \geq \frac{(2^n \P(\tilde{E}_v \setminus \tilde{F}_v))^2}{C \sqrt{n} 2^n \P(E_v)} \geq \frac{1}{C' \log^5 n}\,,$$
for a constant $C'>0$. Recalling the definition of $\tilde{E}_v$, we
complete the proof of the proposition.
\end{proof}

Next, we bootstrap the above estimate and prove the main result in
this section.

\begin{proof}[\emph{\textbf{Proof of Theorem~\ref{thm-lower-bound}.}}]
    Throughout the proof, we write $t=t^-$.
Let $n_1 = 30\log \log n$, $n_3 = \frac{\log^4 n}{\sqrt{\log 2}}$,
$n_4 = 10 (\log \log n)^8$, and $n_2 = n- n_1 - n_3 - n_4$. For
$k\in \N$, write $b_k = \sqrt{\log 2} k - \frac{\log k}{2\sqrt{\log
2}}$. Note that $\sqrt{t} + 50 n_1 \leq b_{n_2} + b_{n_3}$. Our
proof is divided into 4 steps.

\noindent{\bf Step 1}. Write $t_1 = (\sqrt{t} + 2 n_1)^2$. Since for
all $v\in V_{n_1}$ we have $L^{v}_{\tau(t)} \sim
\mathrm{PoiGamma}(t/n_1, n_1)$, an application of
\eqref{eq-upper-tail-bis} yields that
\begin{equation}\label{eq-lower-1}
\P(\exists v\in V_{n_1}: L^{v}_{\tau(t)} \geq t_1) \leq 2^{n_1}\P(
\mathrm{PoiGamma}(t/n_1, n_1) \geq t_1) \leq 2^{n_1} \mathrm{e}^{-4
n_1} = o(1)\,.
\end{equation}

\noindent{\bf Step 2}. For $v\in V_{n_1}$, let $T_v$ be the subtree
rooted at $v$ of height $n_2$. Write $t_2 = (\sqrt{t_1} -
b_{n_2})^2$. By \eqref{eq-lower-1}, we assume in what follows that
$L^v_{\tau(t)} \leq t_1$. Applying
Proposition~\ref{prop-lower-bound} to the subtree $T_v$, we deduce
that for a constant $c>0$,
\begin{equation*}
\P(\min_{u\in T_v \cap V_{n_1+n_2}} L^u_{\tau(t)} \leq t_2) \geq
\frac{c}{\log^5 n}\,.
\end{equation*}
Let $S_1 = \{v\in V_{n_1} : \min_{u\in T_v \cap V_{n_2}}
L^u_{\tau(t)} \leq t_2\}$. By independence of the random walk on
different subtrees, we obtain that with high probability $|S_1| \geq
2^{n_1} / \log^6 n \geq 2^{2 \log \log n}$. We assume this in what
follows. Define $S_2 = \{u\in V_{n_1 + n_2}: L^u_{\tau(t)} \leq
t_2\}$. We see that $|S_2| \geq |S_1| \geq 2^{2 \log \log n}$.

\noindent{\bf Step 3}. For $u\in S_2$, consider the subtree $T_u$
rooted at $u$ of height $n_3$.  Since $(\sqrt{t_2} - b_{n_3})^2 \leq
(\log\log n)^8$. We can apply Proposition~\ref{prop-lower-bound}
again to the subtree $T_u$ and obtain that for a constant $c>0$
\begin{equation*}
\P(\min_{w\in T_u \cap V_{n_1 + n_2 + n_3}} L^w_{\tau(t)} \leq
(\log\log n)^8) \geq \frac{c}{(\log \log n)^5}\,.
\end{equation*}
Let $S_3 = \{w\in T_u \cap V_{n_1 + n_2 + n_3}: L^w_{\tau(t)} \leq
(\log\log n)^8\}$. We can then obtain that with high probability
$|S_3| \geq 2^{\log \log n}$, and we assume this in what follows.

\noindent{\bf Step 4}. For $w\in S_3$, let $T_w$ be the subtree
rooted at $w$ that contains all its descendants. We trivially have
that
$$\P(\min_{w'\in T_w} L^{w'}_{\tau(t)} = 0) \geq \frac{1}{2}\,.$$
Since $|S_3| \geq 2^{\log \log n}$, we see that with high
probability there exists a vertex $w'\in \cup_{w\in S_3}T_w
\subseteq V$ with  $L^{w'}_{\tau(t)} = 0$, completing the proof.
\end{proof}

\subsection{Concentration of inverse local time}

We have been measuring the cover time via the inverse local time so
far. In this subsection, we prove that the inverse local is
well-concentrated around the mean and thus it indeed yields a good
estimate on the cover time.

\begin{lemma}\label{lem-concentration-inverse-time}
Consider a random walk on a rooted binary tree $T = (V, E)$ of
height $n$. Then,
$$\var(\tau(t)) = O(1) \cdot 2^{2n} t\,.$$
\end{lemma}
\begin{proof}
Note that $\tau(t) = \sum_{v\in V} d_v L^v_{\tau(t)}$. Consider $u,
v\in V$ and write $w = u \wedge v$. Assume that $w\in V_k$. Then,
$$\E(L^v_{\tau(t)} \cdot L^u_{\tau(t)}) = \E(\E(L^v_{\tau(t)} \cdot L^u_{\tau(t)}) \mid L^w_{\tau(t)}) = \E((L^w_{\tau(t)})^2) = t^2 + \var(L^w_{\tau(t)}) \leq t^2 + 16tk\,,$$
where the last transition follows from the fact that
$L^w_{\tau(t)}\sim \mathrm{PoiGamma}(t/k, k)$ and a simple
application of the total variance formula $\var X = \E (\var (X \mid
Y)) + \var(\E(X \mid Y))$. We then have $\mathrm{Cov}(L^v_{\tau(t)},
L^u_{\tau(t)}) \leq 16 tk$. Therefore,
$$\var(\tau(t)) = \sum_{u, v\in V} d_v d_u \mathrm{Cov}(L^v_{\tau(t)}, L^u_{\tau(t)}) \leq  \sum_{k=1}^n 2^k 2^{2(n-k)} 3^2 \cdot 16 tk = O(1) \cdot 2^{2n} t \,,$$
where we used the fact that $d_v \leq 3$.
\end{proof}

Now it is obvious that Theorems~\ref{thm-upper-bound},
\ref{thm-lower-bound} imply Theorem~\ref{thm-cover-inverse}.
Together with Lemma~\ref{lem-concentration-inverse-time}, we
complete the proof of Theorem~\ref{thm-cover-time}, by noting that
$\E(\tau(t)) = t \cdot 2|E| = (2^{n+2}-4) t$.

\section*{Acknowledgement}

We thank Amir Dembo and Yuval Peres for helpful discussions.

\small

\bibliography{covertime}
\bibliographystyle{abbrv}

\end{document}